\documentclass[sigconf]{acmart}

\usepackage{balance}

\def\BibTeX{{\rm B\kern-.05em{\sc i\kern-.025em b}\kern-.08emT\kern-.1667em\lower.7ex\hbox{E}\kern-.125emX}}

\newtheorem{defi}{Definition}

\newtheorem{prop}{Proposition}

\newtheorem{cor}{Corollary}

\usepackage{enumitem}
\setlistdepth{9}

\setlist[itemize,1]{label=$\bullet$}
\setlist[itemize,2]{label=$\bullet$}
\setlist[itemize,3]{label=$\bullet$}
\setlist[itemize,4]{label=$\bullet$}
\setlist[itemize,5]{label=$\bullet$}
\setlist[itemize,6]{label=$\bullet$}
\setlist[itemize,7]{label=$\bullet$}
\setlist[itemize,8]{label=$\bullet$}
\setlist[itemize,9]{label=$\bullet$}

\renewlist{itemize}{itemize}{9}

\copyrightyear{2024}
\acmYear{2024}
\setcopyright{acmcopyright}\acmConference[ISSAC '24]{Proceedings of the 2024 International Symposium on Symbolic and Algebraic Computation}{July 16--19, 2024}{Raleigh, NC, USA}
\acmBooktitle{Proceedings of the 2024 International Symposium on Symbolic and Algebraic Computation (ISSAC '24), July 16--19, 2024, Raleigh, NC, USA}

\begin{document}

\fancyhead{}


\title{Liouvillian Solutions of Third Order Differential Equations}


\author{Camilo Sanabria}
\affiliation{%
  \institution{Universidad de los Andes}
  \streetaddress{Carrera 1 \# 18A – 12}
  \city{Bogota}
  \country{Colombia}}
\email{c.sanabria135@uniandes.edu.co}

\author{Thierry Combot}
\affiliation{%
  \institution{University of Burgundy}
  \streetaddress{6 avenue Alain Savary}
  \city{Dijon}
  \country{France}}
\email{thierry.combot@u-bourgogne.fr}

%

%
\begin{abstract} 
Consider a third order linear differential equation $L(f)=0$, where $L\in\mathbb{Q}(z)[\partial_z]$. We design an algorithm computing the Liouvillian solutions of $L(f)=0$. The reducible cases devolve to the classical case of second order operators, and in the irreducible cases, only finitely many differential Galois groups are possible. The differential Galois group is obtained through optimized computations of invariants and semi-invariants, and if solvable, the solutions are returned as pullbacks and gauge transformations of algebraic generalized hypergeometric function ${}_3F_2$. The computation time is practical for reasonable size operators.
\end{abstract}

%
%



%

%
\maketitle

\section{Introduction}

We consider a linear differential equation of third order
\begin{equation}\label{eq0}
a_3(z)\frac{d^3 }{d z^2} y(z)+a_2(z)\frac{d^2 }{d z^2} y(z)+a_1(z)\frac{d }{d z} y(z)+a_0(z)y(z)=0
\end{equation}
where $a_i\in\mathbb{Q}[z]$ are globally coprime and $a_3\ne 0$. The differential Galois group is the Galois group of the differential automorphisms of the extension $[K:\mathbb{C}(z)]$ where $K$ is the Picard Vessiot field, i.e. the field generated by a basis of solutions of equation \eqref{eq0}. Particularly interesting solutions of equation \eqref{eq0} are Liouvillian solutions, i.e. those which can be represented using iteratively integrations, exponentiations and algebraic extractions. The Galois groups stabilizes the vector space of Liouvillian solutions, and acts on them as a virtually solvable group. Our objective in this article is to present an effective algorithm to compute these Liouvillian solutions for \eqref{eq0}.

Equation \eqref{eq0} can also be represented as $Lf=0$, where $L\in\mathbb{Q}(z)[\partial_z]$ is the operator
\begin{equation}\label{op0}
  L=a_3\partial_z^3+a_2\partial_z^2+a_1\partial_z+a_0.
\end{equation}
The computation of the differential Galois group is delicate, as we do not have direct access to the Picard Vessiot field $K$ on which $\hbox{Gal}(L)$ acts. As done for second order operators \cite{KOVACIC1986}, \cite{SINGER1993}, \cite{ULMER1996}, we will use the fact \cite{VANDERPUT2003} that $\hbox{Gal}(L)$ is an algebraic subgroup of $GL_3(\mathbb{C})$.  Using a complete classification of them \cite{SINGER1993}, allows to split the analysis in several parts.

\begin{defi}
Let $G\subseteq GL_n(\mathbb{C})$ be a linear group acting on a $\mathbb{C}$-vector space $V$ of dimension $n$. We say that $G$ is \emph{irreducible} if there are no $G$-invariant subspaces of $V$ other than $V$ and ${0}$. We say that $G$ is \emph{imprimitive} if there exist non-trivial subspaces $V_1, V_2,\ldots, V_k$ of $V$, with $k>1$, such that $V=V_1\oplus\ldots\oplus V_k$ and, for each $\sigma\in G$, the mapping $V_i\mapsto \sigma{V_i}$ is a permutation of the set $\{V_1,\ldots,V_k\}$. We call $G$ \emph{primitive} if it's irreducible but is not imprimitive. 
\end{defi}

The Galois group is a priori a subgroup of $GL_3(\mathbb{C})$, however a reduction of the Wronskian is always possible. By substitution of the unknown
$$y(z) \leftarrow y(z) e^{-\int \frac{a_2(z)}{3 a_3(z)} dz},$$
the new operator $L'$ is such that $a_2(z)=0$. As this transformation uses Liouvillian functions, this maps the space of Liouvillian solutions of $L$ to Liouvillian solutions of $L'$. This also ensures that the Wronskian of $L'$ is constant. Thus, any differential automorphism $\sigma \in \hbox{Gal}(L')$ should preserve this Wronskian, and thus $\det \sigma =1$. So $\hbox{Gal}(L')\subset SL_3(\mathbb{C})$, and in the following we can always restrict ourselves to this case.

Now the analysis splits in three cases
\begin{itemize}
\item Reducible groups. In this case, the operator $L$ factorizes as $L=L_1L_2$ or $L=L_2L_1$, where $L_i$ is of order $i$. If $L=L_2L_1$, then a hyperexponential solution
$$y(z)=e^{\int g(z) dz},\quad g\in\mathbb{C}(z)$$
exists. Such solution can be algorithmically found \cite{VANHOEIJ1997}. Conversely, if $L=L_1L_2$, then the adjoint factorizes \cite{SINGER1993} $L^\dagger=L_2^\dagger L_1^\dagger$. Again, such factorization can be detected by looking for hyperexponential solutions. Then, the space of Liouvillian solutions for a reducible $L$ has dimension $0$ if $L=L_1L_2$ and $L_2$ has no Liouvillian solutions, dimension $1$ if $L=L_2L_1$ and $L_2$ has no Liouvillian solution, and dimension $3$ when both factors have Liouvillian solutions. Kovacic's algorithm \cite{KOVACIC1986} produces a basis for the space of Liouvillian solutions of second order operators and, thus, can be used, together with reduction of order, to obtain a basis for the space of Liouvillian solutions of reducible third order differential operators.
\item Imprimitive groups. The only possible decomposition is $\mathbb{C}^3=V_1\oplus V_2 \oplus V_3$ with $\dim V_i=1$. If requiring moreover that $L$ is irreducible, then only $4$ groups $G_{27}\subset G_{54},G_{81}\subset G_{162}$, all finite, are possible \cite{ULMER2003}. They are generated by a cyclic permutation matrix $T$ and by diagonal matrices of orders dividing $9$. As the operator is irreducible and the Galois group is finite, all solutions of equation \eqref{eq0} are algebraic, and thus Liouvillian.
\item Primitive groups. There are $10$ possible groups, with two, $SL_3$ and $SO_3$, that are not virtually solvable. By irreducibility, if the Galois group is $SL_3$ or $SO_3$, then there are no Liouvillian solutions. The remaining $8$ primitive groups are finite, which lead to a full basis of algebraic solutions (thus Liouvillian). These $8$ groups are
  \begin{itemize}
    \item The Klein group $G_{168}$ and its direct product with the cyclic group of order three $G_{168}\times C_{3}=\langle R,S,T,Z\rangle$.
    \item The $3$-fold covering of the Hessian group in $SL^3$, $H_{216}^{SL_3}$, and it's subgroups $H_{72}^{SL_3}$ and $F_{36}^{SL_3}$.
    \item The Valentiner group $A_6^{SL_3}$.
    \item The alternating group $A_5$ and its direct product with the cyclic group of order three $A_5\times C_3$.
 \end{itemize} 
\end{itemize}
We say that $L$ is \emph{Fuchsian} if any point $z_0$ (including infinity) is \emph{regular}, i.e. there exists a basis of solution of formal series in $(z-z_0)^\alpha \mathbb{C}[[z-z_0]][\ln(z-z_0]$ for a certain $\alpha\in\mathbb{C}$. Such condition can be easily verified, and leads to an effective criterion for existence of Liouvillian solutions.

\begin{cor}\label{cor1}
The solvable irreducible algebraic subgroups of $SL_3(\mathbb{C})$ are all finite, therefore an irreducible third order differential operator with Galois group in $SL_3(\mathbb{C})$ has Liouvillian solutions if and only if all its solutions are algebraic.
\end{cor}

The corollary implies that an irreducible third order linear operator with Liouvillian solutions is Fuchsian, and even more, has a basis of Puiseux series solutions at every point.

\section{Semi-Invariants}

\subsection{List of Semi-Invariants}

After performing factorization tests on $L$, we can assume that $L$ is irreducible, and thus the existence of Liouvillian solutions devolves to testing whether the differential Galois group is one of the above $12$ finite groups. So from now on, we can assume that $L$ is irreducible and $G$ is solvable (else there are no Liouvillian solutions to find).

From here on, $y_1,y_2,y_3$ will denote a basis of solutions with Puiseux power series near a point $z_0$, $y_1,y_2,y_3\in\mathbb{Q}[[(z-z_0)^{1/r}]]$. In practice, we wish to use for $z_0$ a singularity of equation \eqref{eq0} (i.e. such that $a_3(z_0)=0$) for which the Puiseux series solutions require the largest $r\in\mathbb{N}$. A basis of solutions induces a representation $G\subseteq SL_3(\mathbb{C})$ of the differential Galois group of $L$.

\begin{defi}
Given three polynomials $P_1,P_2,P_3\in\mathbb{C}[X_1,X_2,X_3]$, we denote the Hessian, bordered Hessian and Jacobian by
\begin{align*}
H(P_1) & =  \hbox{det}\left( \left(\partial_{i,j} P_1\right)_{i,j=1\dots 3} \right)\\
HB(P_1,P_2) & =  \hbox{det}\left( \begin{array}{cc} \left(\partial_{i,j} P_2\right)_{i,j=1\dots 3} & (\partial_i P_1)_{i=1\dots 3}^\intercal\\ ((\partial_i P_1)_{i=1\dots 3}) & 0 \end{array} \right)\\
J(P_1,P_2,P_3) & =  \hbox{det}\left( \left(\partial_{i} P_j\right)_{i,j=1\dots 3} \right).\\
\end{align*}
\end{defi}

Given a finite group $G\subset SL_3(\mathbb{C})$, we consider the set of semi-invariants
$$\{P\in\mathbb{C}[X_1,X_2,X_3], \forall g \in G, \exists \alpha\in\mathbb{C}^*, P\circ g=\alpha P\}.$$
Given a semi-invariant $P\in\mathbb{C}[X_1,X_2,X_3]$, we define the \emph{character} of $P$ as the morphism
$$\chi: G \rightarrow \mathbb{C}^*\quad g \rightarrow P\circ g/P.$$
The \emph{degree of $P$} is the total degree of $P$, and \emph{the order of $P$} is $\sharp \hbox{Im}(\chi)$.

These semi invariants generate a finitely generated algebra $\mathcal{A}$, which is, respectively for our $12$ irreducible finite groups:
\begin{itemize}
\item For $G_{168}$ and $G_{168}\times C_3$, $\mathcal{A}=\mathbb{C}[F_4,F_6,F_{14},F_{21}]$ where
$$F_4=X_1^3X_2+X_2^3X_3+X_3^3X_1,$$
$$54F_6=H(F_4), 9F_{14}=HB(F_4,F_6),14F_{21}=J(F_4,F_6,F_{14}).$$
\item For $A_6$, $\mathcal{A}=\mathbb{C}[F_6,F_{12},F_{30},F_{45}]$ where 
$$F_6=9 X_1^5 X_3+10 X_1^3 X_2^3-45 X_1^2 X_2^2 X_3^2-135 X_1 X_2 X_3^4+9 X_2^5 X_3+27 X_3^6,$$
$$20250F_{12}=-H(F_6), 24300F_{30} = HB(F_6,F_{12}), 4860 F_{45} = J(F_6,F_{12},F_{30}).$$
\item For $A_5$ and $A_5\times C_3$, $\mathcal{A}=\mathbb{C}[F_2,F_6,F_{10},F_{15}]$ where $F_2=X_1^2+X_2X_3$,
$$F_6=8  X_1^4  X_2  X_3-2  X_1^2  X_2^2  X_3^2- X_1 ( X_2^5+ X_3^5)+ X_2^3  X_3^3,$$
$$F_{10}=HB(F_2,F_6)-32F_2^2F_6,10F_{15}=-J(F_2,F_6,F_{10}).$$
\item For $G_{27}$, $G_{54}$, $G_{81}$ and $G_{162}$, $\mathcal{A}=\mathbb{C}[P_3,S_3,Q_6,R_9]$ where
$$P_3=X_1X_2X_3, S_3=X_1^3+X_2^3+X_3^3,$$
$$Q_6=X_1^3X_2^3+X_1^3X_3^3+X_2^3X_3^3,R_9 = (X_1^3-X_2^3)(X_1^3-X_3^3)(X_2^3-X_3^3).$$
\item For $F_{36}^{SL_3}$, $\mathcal{A}=\mathbb{C}[F_3,\Phi_3,F_6,R_9]$ where $F_3=-3(1+\sqrt{3})P_3+S_3$,
$$\Phi_3=-3(1-\sqrt{3})P_3+S_3, 648 F_6=HB(F_3,\Phi_3)+HB(\Phi_3,F_3).$$
\item For $H_{72}^{SL_3}$, $\mathcal{A}=\mathbb{C}[\Phi_6,F_6,R_9,\Phi_{12}]$, where $\Phi_6=F_3\Phi_3$,
$$108000 \Phi_{12}=-H(F_6).$$
\item For $H_{216}^{SL_3}$, $\mathcal{A}=\mathbb{C}[F_6,R_9,\Phi_{12},F_{12}]$ where
$$F_{12}=HB(F_6,R_9)/(675\Phi_{12}).$$

\end{itemize}

\subsection{Symmetric powers}

\begin{defi}
Let $A\in M_3(\mathbb{C})$. We denote $\hbox{Sym}^m(A)\in M_N(\mathbb{C})$, $N=\binom{3+m-1}{m-1}$, the $m$-th symmetric power of $A$ obtained by the following procedure
\begin{enumerate}
\item Let $c_i:=X_1a_{1i}+X_2a_{2i}+X_3a_{3i}\in\mathbb{C}[X_1,X_2,X_3],\; i=1\dots 3$. 
\item For $i,j$ from $1$ to $N$ the $ij$-th entry of $\hbox{Sym}^m(A)$ is the coefficient on the monomial $\mathbf{X}^{\mathbf{m}_{i}}/\binom{m}{\mathbf{m}_i}$ of the polynomial $\mathbf{c}^{\mathbf{m}_j}/\binom{m}{\mathbf{m}_j}$, where $\mathbf{m}_k$ is the $k$th element of $\{[m_1,m_2,m_3]|\ m_1+m_2+m_3=m\}$, ordered in reversed lexicographic order on the exponents (i.e. $[m,0,0]<[m-1,1,0]<\ldots<[0,0,m]$).
\end{enumerate}
The $m$-th symmetric power of $L$ is the operator $\hbox{Sym}^m(L)$ whose space of solutions are the homogeneous polynomials of degree $m$ in $y_1,y_2,y_3$.
\end{defi}

The matrix system associated to equation \eqref{eq0} is $X'=AX$ where $A$ is the companion matrix
$$\left[ \begin{array}{rrr} 0 & 1 & 0\\ 0 & 0 & 1\\ -a_0/a_3 & -a_1/a_3 & -a_2/a_3 \end{array}\right].$$
As $y_1,y_2,y_3$ is a basis of solutions of \eqref{eq0}, then its Wronskian
$$Y=\left[\begin{array}{rrr} y_1 & y_2 & y_3\\ y_1' & y_2' & y_3'\\  y_1'' & y_2'' & y_3'' \end{array}\right]$$
is the fundamental matrix of the associated matrix system. Two matrix systems $X'=AX$ and $X'=\tilde{A}X$ are \emph{gauge equivalent} over a differential field $E\supseteq\mathbb{Q}(z)$ if there exists a matrix $M$ with coefficients in $E$ such that $MY$ is a fundamental matrix of the latter whenever $Y$ is a fundamental matrix of the former, or equivalently $\tilde{A}=M'M^{-1}+MAM^{-1}$.

Let $S^m(A)$ be the matrix such that $\hbox{Sym}^m(Y)'=S^m(A)\hbox{Sym}^m(Y)$ where $\hbox{Sym}^m(Y)$ is the $m$th symmetric power of $Y$. The system with matrix $S^m(A)$ is gauge equivalent over $\mathbb{Q}(z)$ to the system associated to $\hbox{Sym}^m(L)$. So there is a bijection between rational solutions of $\hbox{Sym}^m(L)$ to vectors of constants $C_o$ such that $\hbox{Sym}^m(Y)C_o$ has rational entries.

\begin{prop}
Let $m$ be a positive integer, then $X'=S^m(A)X$ admits $d$ linearly independent hyperexponential solutions if and only if, in the algebra of semi-invariants of $\hbox{Gal}(L)$, the homogeneous semi-invariant polynomials of degree $m$ span a vector space of dimension $d$.
\end{prop}

\begin{proof}
Since the system $X'=S^m(A)X$ and the system associated to $\hbox{Sym}^m(L)$ are gauge equivalent over $\mathbb{Q}(z)$, there exists an invertible matrix $M$ with rational coefficients which sends solutions of one to another. The vector space of hyperexponential solutions of the system $X'=S^m(A)X$ is therefore in bijection with the vector space of hyperexponential solutions of $\hbox{Sym}^m(L)$. Now, the semi-invariants of $Gal(L)$ correspond to the polynomials fixed up to multiplication by $Gal(L)$. Thus, $P(X_1,X_2,X_3)$ is a semi-invariant if and only if the logarithmic derivative of $P(y_1,y_2,y_3)$ is a rational function, equivalently, $P(y_1,y_2,y_3)$ is a hyperexponential solution of $\hbox{Sym}^m(L)$ which, via gauge equivalence, corresponds to a hyperexponential solutions of $X'=S^m(A)X$. Thus, the number of linearly independent hyperexponential solutions of $X'=S^m(A)X$ coincides with the dimension of the vector space of homogeneous semi-invariant polynomials of degree $m$
\end{proof}

Note that, as the solutions of $\hbox{Sym}^m(L)$ are exactly the homogeneous polynomials in $y_1,y_2,y_3$ of degree $m$, a function $\sqrt[r]{p(z)}q(z)$, $p(z)\in\mathbb{C}(z)$, $q(z)\in\mathbb{C}[z]$ is a hyperexponential solution of $\hbox{Sym}^m(L)$ if and only if there exists a semi-invariant homogeneous polynomial $P(X_1,X_2,X_3)$ of degree $m$ and order $r$ such that $P(y_1,y_2,y_3)=\sqrt[r]{p(z)}q(z)$. We have that
$$q(z) = P(p(z)^{-1/{rm}}y_1, p(z)^{-1/{rm}}y_2, p(z)^{-1/{rm}}y_3),$$
and thus $q(z)$ is a rational solution of $\hbox{Sym}^m(L_0)$, where $L_0 = \hbox{Sym}(L,rmp(z)Dz+p'(z))$. Now, $P(X_1,X_2,X_3)$ is invariant under $Gal(L_0)$, and thus, the computation of the semi-invariants of $L$ of degree $m$ and order $r$ reduces to the identification of the candidates of the \emph{hyperexponential part} $\sqrt[r]{p(z)}$ for which $L_0$ could admit an invariant of degree $m$.

Consider the set of roots of $a_3\in\mathbb{Q}[z]$, the generated splitting field $\mathbb{L}$ and its Galois group $G'$. All semi-invariants can be computed with coefficients in $\mathbb{L}$, as finding them devolves to compute hyperexponential solutions of some symmetric power $\hbox{Sym}^m(L)$ \cite{WEIL1995}.  Let us now consider the semi-invariants of given degree $m$ and order dividing $r$. These generate a vector space $\mathcal{V}_{m,r}$ of dimension $d$ of solutions of $\hbox{Sym}^m(L)$ that are linear combinations of hyperexponential solutions of the form $\sqrt[r]{p(z)}q(z)$.

\begin{prop}\label{prop_d}
Let $y_1,y_2,y_3$ be a basis of solutions of $L$. Let $P$ be a semi-invariant of $Gal(L)$ of degree $m$ and order $r$. Then, up to multiplication by a constant, we have $P(y_1,y_2,y_3)=\sqrt[r]{p(z)}q(z)$ where $p(z)$ has coefficients in an extension of $\mathbb{Q}$ of degree at most $\dim  \mathcal{V}_{m,r}$.
\end{prop}

\begin{proof}
Let $G_0$ be the absolute Galois group $Gal(\bar{\mathbb{Q}}:\mathbb{Q})$. From lemma 2.1 in \cite{VANDERPUT1993}, the group of differential automorphisms of $K\supset\mathbb{Q}(z)$ is the semi-direct product $Gal(L)\rtimes G_0$. Given a semi-invariant $P$ of degree $m$, order $r$ and character $\chi$, $P(y_1,y_2,y_3)=\sqrt[r]{p(z)}q(z)$ is a solution of $\hbox{Sym}^m(L)$. Since every $\sigma\in G_0$ commutes with $\partial_z$, then $\sigma\left(P(y_1,y_2,y_3)\right)=\sqrt[r]{p_\sigma(z)}q(z)$ is a solution of $\hbox{Sym}^m(L)$, and therefore there exists a semi-invariant $P_{\sigma}$ of degree $m$ such that $P_{\sigma}(y_1,y_2,y_3)=\sqrt[r]{p_\sigma(z)}q(z)$. Let $\chi_\sigma$ be its character. Given $g\in Gal(L)$ let $g^{\sigma}\in Gal(L)$ be such that $\sigma g=g^\sigma \sigma$ in $Gal(L)\rtimes G_0$. We have
\begin{align*}
  \chi_{\sigma}(g^{\sigma})P_\sigma(y_1,y_2,y_3) & = g^{\sigma}\left(P_\sigma(y_1,y_2,y_3)\right) = g^{\sigma}\sigma\left(P(y_1,y_2,y_3)\right) \\
  & = \sigma g\left(P(y_1,y_2,y_3)\right)=\sigma\left(\chi(g)P(y_1,y_2,y_3)\right).
\end{align*}
Therefore $\chi_{\sigma}(g^{\sigma})=\sigma(\chi(g))$ and $P_\sigma$ is a semi-invariant of order $r$. Since to each image of $p(z)$ under $G_0$ corresponds a solution of $\hbox{Sym}^m(L)$ and a semi-invariant of degree $m$ and order $r$, then the orbit of $p(z)$ under $G_0$ has at most $d=\dim \mathcal{V}_{m,r}$ elements, and thus its coefficients lie in an extension of $\mathbb{Q}$ of degree at most $d$.
\end{proof}

\subsection{Computing semi-invariants}

To describe the algorithm to find the semi-invariants we introduce some notation.

Given a set $S\subseteq\mathbb{Q}$ and a positive integer $m$, we denote by $mS$ the set of $m$-fold sums of elements of $S$. Given a positive integer $r$, we denote by $s\sim_r t$ the relation in $S$ defined by $r(s-t)\in\mathbb{Z}$, and by $[S]_r$ a set of representative of $\sim_r$ in $S\cap\frac{1}{r}\mathbb{Z}$.

Given three integers $\mathbf{m}=[m_1,m_2,m_3]$, we denote by $Sym^{\mathbf{m}}(L)$ the differential operator with solutions $y_1^{m_1}y_2'^{m_2}y_3''^{m_3}$ whenever $y_1$, $y_2$, $y_3$ are solutions of $L$.

We denote $W$ the generic Wronskian matrix
$$W=\left[\begin{array}{rrr} X_1 & X_2 & X_3\\ X_1' & X_2' & X_3'\\  X_1'' & X_2'' & X_3'' \end{array}\right]$$
where $X_1,X_2,X_3$ are generic functions.

Our algorithm \underline{\sf semiInvariants} produces a set of generators of the semi-invariants. It uses van-Hoeij and Weil's algorithm from \cite{VANHOEIJWEIL1997} that computes invariants and present them in the following way. Given an equation \eqref{eq0}, a positive integer $m$ and a basis of solutions $\mathbf{y}=(y_1,y_2,y_3)$ in power series, the algorithm produces a collection of constant vectors $C_{o,1}$, $\ldots$, $C_{o,s}$ such that the vectors $\hbox{Sym}^m(Y)C_{o,1}$, $\ldots$, $\hbox{Sym}^m(Y)C_{o,s}$ form a basis of rational solutions, and $\hbox{Sym}^m(W)C_{o,1}$, $\ldots$, $\hbox{Sym}^m(W)C_{o,s}$ form a basis of invariant, of the system $X'=S^m(A)X$.

\noindent\underline{\sf semiInvariants}\\
\textsf{Input:} An equation \eqref{eq0}, two positive integers $m,r$, and $d\in\{1,2\}$.\\
\textsf{Output:} A collection of couples $\left(C_{o,1},\sqrt[r_1]{p_1(z)}\right)$, $\ldots$, $\left(C_{o,s},\sqrt[r_s]{p_s(z)}\right)$ of constant vectors and radical of polynomials such that, for some basis of solutions $y_1,y_2,y_3$, the vectors $\hbox{Sym}^m(Y)C_{o,1}$, $\ldots$, $\hbox{Sym}^m(Y)C_{o,s}$ generate the hyperexponential solutions,  $\sqrt[r_1]{p_1(z)}$, $\ldots$, $\sqrt[r_s]{p_s(z)}$ are their respective hyperexponential parts, and $\hbox{Sym}^m(W)C_{o,1}$, $\ldots$, $\hbox{Sym}^m(W)C_{o,s}$ generate the semi-invariants of the system $X'=S^m(A)X$ of order dividing $r$, and coefficients in an extension of $\mathbb{Q}$ of degree $\le d$.\\
\begin{enumerate}[leftmargin=*]
\item Let $y_1,y_2,y_3$ be a basis of solutions of \eqref{eq0} in power series.
\item For $d=1$, note $q_i,\;1\le i\le r_1$ the irreducible factors of $a_3\in\mathbb{Q}[z]$. For $d=2$, note $q_i$ the irreducible factors of $a_3\in\bar{\mathbb{Q}}[z]$.
\item For each $1\le i\le r_1$, consider a root of $q_i$, and compute the set of local exponents of this root, giving sets $E_{q_1},\ldots,E_{q_{r_1}}$.
\item Let $S$ be the set of polynomials
$$p(z)=\prod_i^{r_1} q_i(z)^{re_i},$$
$e_i\in [E_{q_i}]_r$, such that $p(z)$ has coefficients in a field extension of $\mathbb{Q}$ of degree $\le d$.
\item $\hbox{Out}:=[]$.
\item For each $p(z)\in S$, compute the operator $L_0=\hbox{Sym}(L,rmp(z)Dz+p'(z))$ and set
$$\hbox{Out}:=\hbox{Out}\cup[(C_{o,1},p(z)),\ldots, (C_{o,t},p(z))]$$
where $[C_{o,1},\ldots,C_{o,t}]$ are constant vectors obtained from van Hoeij and Weil's algorithm \cite{VANHOEIJWEIL1997} for $L_0$, $m$ ,and $\mathbf{y}/\sqrt[r]{p(z)}$.
\item \textsf{Return} $\hbox{Out}$.
\end{enumerate}

Note that, an inspection of the possible algebras $\mathcal{A}$ generated by the semi-invariants reveals that we will only use $d=2$ in \underline{\sf semiInvariants} with $m=3$ when probing for semi-invariants of the imprimitive groups (with $r=2$) and of $F_{36}^{SL_3}$ ($r=4$). For the rest of semi-invariants it will suffice to use $d=1$.

Given an equation \eqref{eq0} and a triple of non-negative integers $\mathbf{m}=[m_1,m_2,m_3]$ one can obtain a rational function $Q(z)$ and a non-negative integer $N$ such that every rational solution of $\hbox{Sym}^{\mathbf{m}}(L)$ is of the form $Q(z)q(z)$ where $q(z)$ is a polynomial of degree not bigger than $N$. How to compute the bounds $Q(z)$ and $N$ algorithmically is explained in \cite{VANHOEIJWEIL1997}. Given $Q(z)$ and $N$, getting the explicit expression of the vectors of rational functions $\hbox{Sym}^m(Y)C_{o}$ from the power series of $y_1,y_2,y_3$ at $z_0$, devolves to finding $q(z)\in\mathbb{Q}[x]_{\le N}$ such that 
$$Q_{\mathbf{m}}(z)q(z)\equiv \hbox{Sym}^m(Y)C_{o}\pmod{(z-z_0)^{N}}.$$
In particular, there is a bound on the number of terms to consider in the power series expansion of the solutions $y_1,y_2,y_3$.

\noindent\underline{\sf valuesOfSemiInvariant}\\
   \textsf{Input:} An equation \eqref{eq0}, two positive integers $m,r$, a basis of solutions $\mathbf{y}=(y_1,y_2,y_3)$ in power series, a semi-invariant $P(X_1,X_2,X_3)$, and the hyperexponential part $\sqrt[r]{p(z)}$ of $P(\mathbf{y})$.\\
   \textsf{Output:} The entries of $F = \hbox{Sym}^m(Y)C_o$, where  $C_o$ is the vector of coefficients of $P$ in the monomials $\mathbf{X}^{\mathbf{m}_{i}}/\binom{m}{\mathbf{m}_i}$, in the form hyperexponential part times a polynomial.\\
   \begin{enumerate}[leftmargin=*]
      \item Obtain the vector $C_o$ coefficients of $P$ in the monomials $\mathbf{X}^{\mathbf{m}_{i}}/\binom{m}{\mathbf{m}_i}$
      \item Let $Y_0$ be the Wronskian of $\sqrt[rm]{p(z)}[y_1,y_2,y_3]$ and let $N=\binom{3+m-1}{3-1}$.
      \item Compute the operator $L_0 := \hbox{Sym}(L,rmP(z)Dz+p'(z))$.
      \item For every $\mathbf{m}$ in the collection of triples $[m_1,m_2,m_3]$ such that $m_1+m_2+m_3=m$, obtain the bounds $\left[Q_{\mathbf{m}}(z), N_{\mathbf{m}}\right]$  for $\hbox{Sym}^{\mathbf{m}}(L_0)$.
      \item For $i$ from $1$ to $N$ find $q_i(z)\in\mathbb{Q}[x]_{\le N_{\mathbf{m}_i}}$ such that 
      $$Q_{\mathbf{m}_i}q_i(z)\equiv \hbox{Sym}^m(Y_0)_iC_{o,i}\pmod{(z-\alpha)^{N_{\mathbf{m}_i}}}$$
      with $\alpha\in\mathbb{C}$, and set $F_i :=Q_{\mathbf{m}_i}q_i(z)$.
      \item Let $h(z)=\sqrt[mr]{p(z)}$ and
      $$C_h:=\left[   \begin{array}{lll}      h(z) & 0 & 0\\    h'(z) & h(z) & 0\\     h''(z) & 2h'(z) & h(z)  \end{array}\right].$$
      \item \textsf{Return} $F:=\hbox{Sym}^m(C_h)[F_i]$.
   \end{enumerate}

An inspection of the algorithm \underline{\sf valuesOfSemiInvariant} shows that it suffices to take the expansion up to order $\max_i N_{\mathbf{m}_i}+3-1$ when considering the power series expansion of the solutions at a point $z_0\ne \infty$. This bound can be computed directly from \eqref{eq0} and the integer $m$, before running the algorithms \underline{\sf semiInvariant} and \underline{\sf valuesOfSemiInvariant}.

\subsection{Riccati Polynomials and Pullbacks}

A polynomial $R(T)\in\mathbb{C}(z)[T]$ is a Riccati polynomial of \eqref{eq0} if $e^{\int\omega(z)dz}$ is a solution of \eqref{eq0} whenever $\omega(z)$ is a root of $R(T)$, or equivalently, the logarithmic derivative of a solution is a root of $R(T)$.

Given a homogeneous semi-invariant polynomial $P(X_1,X_2,X_3)\in\mathbb{C}[X_1,X_2,X_3]$ that factors into linear terms, Theorem 2.1 in \cite{VANHOEIJ1999} implies that the following algorithm produces a Riccati polynomial.

\noindent\underline{\sf produceRiccatiPolynomial}\\
\textsf{Input:} An equation \eqref{eq0}, a basis of solutions $y_1,y_2,y_3$ in power series, and a homogeneous semi-invariant polynomial $P(X_1,X_2,X_3)$ that factors into linear terms.\\
\textsf{Output:}  A Riccati polynomial for $L$.\\
\begin{enumerate}[leftmargin=*]
   \item Let $C_o$ be the vector of coefficients of $P(X_1,X_2,X_3)$ in the monomials $\mathbf{X}^{\mathbf{m}_{i}}/\binom{m}{\mathbf{m}_i}$ and $m=\deg(P)$.
   \item Obtain $F = \hbox{Sym}^m(Y)C_o$ with \underline{\sf valuesOfSemiInvariant} and denote $a$ the first entry of $F$.
   \item Let
   $$C_T=\left[ \begin{array}{rrr} T & -1 & 0\\ 0 & 1 & 0\\ 0 & 0 & 1\end{array}\right].$$
   \item \textsf{Return} the first entry of $\frac{1}{a}\hbox{Sym}^m(C_T)F$.
\end{enumerate}

Indeed, the entries of the first row of $\hbox{Sym}^m(C_T)\hbox{Sym}^m(Y)=\hbox{Sym}^m(C_TY)$ are $Ty_i-y'_i$ and so the first entry of $\frac{1}{a}\hbox{Sym}^m(C_T)F$ is $P(T\mathbf{y}-\mathbf{y}')/P(\mathbf{y})$, a polynomial with rational coefficients and roots $y'_i/y_i$.

Consider another linear differential equation of third order
\begin{equation}\label{eqb}
b_3(z)\frac{\partial^3 }{\partial z^2} f(z)+b_2(z)\frac{\partial^2 }{\partial z^2} f(z)+b_1(z)\frac{\partial }{\partial z} f(z)+b_0(z)f(z)=0
\end{equation}
where $b_i\in\mathbb{C}[z]$. Equation \eqref{eqb} is a \emph{pullback} of \eqref{eq0} if there exist a function $s(z)$ such that $f(z)=y(s(z))$ is a solution of \eqref{eqb} whenever $y(z)$ is a solution of \eqref{eq0}.

From \cite{SANABRIA2022}, to describe $\eqref{eq0}$ as the pullback of a ${}_3F_2$ equation operator we need to compute the values of $P(\mathbf{y}+f\mathbf{y}')$ where $P(X_1,X_2,X_3)$ belong to a set of semi-invariant polynomial of degree $m$.

   \noindent\underline{\sf gaugeEquivalentValue}\\
   \textsf{Input:} An equation \eqref{eq0}, a basis of solutions $y_1,y_2,y_3$ in power series, a semi-invariant homogeneous polynomial $P(X_1,X_2,X_3)$ together with its order $r$, the hyperexponential part $\sqrt[r_0]{p(z)}$ of $P(y_1,y_2,y_3)$, and a function $f$.\\
   \textsf{Output:}  $Q(f)=P(y_1+fy_1',y_2+fy_2',y_3+fy_3')$.\\
   \begin{enumerate}[leftmargin=*]
      \item Let $C_o$ be the vector of coefficients of $P(X_1,X_2,X_3)$ in the monomials $\mathbf{X}^{\mathbf{m}_{i}}/\binom{m}{\mathbf{m}_i}$. Set $m:=\deg(P)$.
      \item Obtain $F = \hbox{Sym}^m(Y)C_o$ from \underline{\sf valuesOfSemiInvariant} with the semi-invariant $P$, degree $m$, and order $r$.
      \item Let
      $$C_f:=\left[ \begin{array}{rrr} 1 & f & 0\\ 0 & 1+f' & f\\ -a_0/a_3f & f''-a_1/a_3f & 1+2f'-a_2/a_3f\end{array}\right].$$ 
      \item \textsf{Return} the first entry of $\hbox{Sym}^m(C_f)F$.
   \end{enumerate}

Note that only the $m+1$st entries of the first row of $\hbox{Sym}^m(C_T)$ and $\hbox{Sym}^m(C_f)$ contain non-zero elements, so to obtain $R(T)$ and $P(y_1+fy_1',y_2+fy_2',y_3+fy_3')$ one only needs the first $m+1$ entries of $F=\hbox{Sym}^m(Y)C_o$ in the second step of the algorithms \underline{\sf produceRiccatiPolynomial} and \underline{\sf gaugeEquivalentValue}.

\section{Primitive and Imprimitive groups}

\subsection{The imprimitive case}

   From the list of the possible semi-invariant algebras $\mathcal{A}$, we see that $\hbox{Gal}(L)$ is imprimitive if and only if $L$ admits a semi-invariant of degree $3$ and order $2$. We identify the semi-invariant that factors into linear terms with the following proposition.

   \begin{prop}\label{prop_linear}
      Let $P_3=X_1X_2X_3$ and $S_3 = X_1^3+X_2^3+X_3^3$ then $P=aP_3+bS_3$ factors into linear factors if and only if $HB(P)=\lambda P$ for some $\lambda\in\mathbb{C}$.
   \end{prop}

   \begin{proof}
      A polynomial $P$ of degree $3$ factors into three linearly independent linear factors if and only if, up to a linear change of variables, $P=Z_1Z_2Z_3$. Suppose that $\mathbf{X}=\mathbf{Z}C$ where $C$ is a $3\times 3$ matrix with constant coefficients. We may assume that $\det(C)=1$. If $D_{\mathbf{X}}$ denotes the total derivative with respect to $\mathbf{X}$ then, applying the chain rule twice and noticing that $D_{\mathbf{Z}}\mathbf{X}=C$, we obtain $D_{\mathbf{Z}}P=D_{\mathbf{X}} P\cdot D_{\mathbf{Z}}\mathbf{X}=D_{\mathbf{X}}P\cdot C$ and $D^2_{\mathbf{Z}}P=C^T\cdot D^2_{\mathbf{X}}P\cdot C$. So $H(P)=H(P)\det(C)^2=\det\left(D^2_{\mathbf{Z}}P\right)=2P$. Conversely, by inspection, for all $a,b$ such that $HB(P)=2P$ we have that $P$ factors into linear terms.
   \end{proof}

   \noindent\underline{\sf riccatiSolutionImprimitive}\\
   \textsf{Input:} An equation \eqref{eq0} with $a_2(z)=0$ and imprimitive differential Galois group.\\
   \textsf{Output:} A Riccati polynomial for \eqref{eq0}.\\
   \begin{enumerate}[leftmargin=*]
      \item Let $y_1,y_2,y_3$ be a basis of solutions of \eqref{eq0} in power series.
      \item Using \underline{\sf semiInvariants}, obtain the set $B:=\{\left(C_{o,1},\sqrt[r_1]{p_1(z)}\right)$, $\ldots$, $\left(C_{o,s},\sqrt[r_s]{p_s(z)}\right)\}$ of coefficients of the semi-invariants of degree $m=3$ and order $r|2$, together with their hyperexponential parts.
      \item If $\sharp B=1$, then let $P$ be the polynomial with coefficients $C_{o,1}$ in the monomials $\mathbf{X}^{\mathbf{m}_{i}}/\binom{m}{\mathbf{m}_i}$. Else $\sharp B=2$, then let $P_1$ and $P_2$ be the polynomial with coefficients $C_{o,1}$ and $C_{o,2}$ in the monomials $\mathbf{X}^{\mathbf{m}}/\binom{m}{\mathbf{m}}$, find $a_1,a_2\in\mathbb{C}$ such that $H(a_1P_1+a_2P_2)=2(a_1P_1+a_2P_2)$ and let $P:=a_1P_1+a_2P_2$.
      \item \textsf{Return} $R(T)$, computed from {\underline{\sf produceRiccatiPolynomial}} using the semi-invariant $P$.
   \end{enumerate}

\subsection{The primitive case}

   As discussed above, if $\hbox{Gal}(L)\subseteq SL_3(\mathbb{C})$ is irreducible and $L$ has Liouvillian solutions, then all the solutions are algebraic. From Theorem 7 and  corollary 8 in \cite{SANABRIA2022} follows that a fundamental matrix for the system $X'=AX$ is of the form $Y(z)=CY_0(s(z))$ where $Y_0$ is the Wronskian of a basis of solutions of a ${}_3F_2$ equation, and the entries of the $3\times 3$ matrix $C$ and $s(z)$ are in an algebraic extension of $\mathbb{Q}(z)$ of known bounded degree. The specific ${}_3F_2$ operator, and the formulas for $C$ and $s(z)$ depend on $\hbox{Gal}(L)$ and are exposed in \cite{SANABRIA2022}.

   For the following $5$ \underline{\sf pullbackSolution***} algorithms, where \underline{\sf ***} is \underline{\sf G168}, \underline{\sf A6}, \underline{\sf H216}, \underline{\sf H72} \underline{\sf F36} the output is the same:\\
   \textsf{Output:} A ${}_3F_2$ equation $E$, a function $s(z)\in\mathbb{Q}(z,f)$, the minimal polynomial $Q(f)\in\mathbb{Q}(z)$ of $f$, and a matrix $C\in M_3(K)$ where $K$ is a radical extension of $\mathbb{Q}(z,f)$, such that $Y(z)=CY_0(s(z))$ is a Wronskian of solutions of \eqref{eq0} whenever $Y_0(z)$ is a Wronskian of solutions of $E$.

   The \underline{\sf pullbackSolution***} algorithms below use the matrices $C_h$ and $C_f$ as defined in \underline{\sf valuesOfSemiInvariant} and \underline{\sf gaugeEquivalentValue}.

   \noindent\underline{\sf pullbackSolutionG168}\\
   \textsf{Input:} An equation \eqref{eq0} with differential Galois group isomorphic to $G_{168}$ or $G_{168}\times C_3$.
   \begin{enumerate}[leftmargin=*]
      \item Let $y_1,y_2,y_3$ be a basis of solutions of \eqref{eq0} in power series.
      \item Using \underline{\sf semiInvariants}, obtain the set $B_4:=\{\left(C_{o,1},\sqrt[r_1]{p_1(z)}\right)\}$ of coefficients of a semi-invariant of degree $4$, order $r_1|3$, and dimension $d=1$, together with its hyperexponential parts.
      \item Let $P_{o,4}$ be the polynomial with coefficients $C_{o,1}$ in the monomials $\mathbf{X}^{\mathbf{m}_{i}}/\binom{m}{\mathbf{m}_i}$. Let $a$ be a variable and set $P_{a,4}:=aP_{o,4}$, $P_{a,6}:=H(P_{a,4})/54$, $P_{a,14}:=HB(P_{a,4},P_{a,6})/9$ and $P_{a,21}=J(P_{a,4},P_{a,6},P_{a,14})/14$.
      \item Set $a$ such that 
      {\footnotesize
      \begin{align*}
         P_{a,21}^2 = & -2048  P_{a,4}^9  P_{a,6}+22016  P_{a,4}^6  P_{a,6}^3-256  P_{a,14}  P_{a,4}^7-60032  P_{a,4}^3  P_{a,6}^5\\
            & +1088  P_{a,14}  P_{a,4}^4  P_{a,6}^2 +1728  P_{a,6}^7+1008  P_{a,14}  P_{a,4}  P_{a,6}^4-88  P_{a,14}^2  P_{a,4}^2  P_{a,6}+ P_{a,14}^3.
      \end{align*}
      }
      \item Let $f$ be a generic function. For $i=4,6,14$, obtain $Q_i(f)$ from \underline{\sf gaugeEquivalentValue} with the semi-invariant $P_{a,i}$; degree $m=i$; order $r=r_1$,$1$, or $r_1$; and exponential part $\sqrt[r_1]{p_1(z)}$, $1$, or $\sqrt[r_1]{p_1(z)^2}$ respectively, depending on whether $i=4$, $6$, or $14$.
      \item Let $Q(f)$ be an irreducible factor in $\mathbb{Q}(z)[f]$ of the numerator of $Q_4(f)$. Set $s(z):= -[Q_{14}(f)]^3/(1728[Q_{6}(f)]^7)$, and $C:=C_f^{-1}C_h$ with $h=\sqrt[6]{Q_{6}(f)}$
      \item \textsf{Return} The ${}_3F_2(-1/42,5/42,17/42;1/3,2/3| z)$ equation, $s(z)$, $Q(f)$ and $C$.
   \end{enumerate}

   \noindent\underline{\sf pullbackSolutionA6}\\
   \textsf{Input:} An equation \eqref{eq0} with differential Galois group isomorphic to $A_6^{SL_3}$.
   \begin{enumerate}[leftmargin=*]
      \item Let $y_1,y_2,y_3$ be a basis of solutions of \eqref{eq0} in power series.
      \item Using \underline{\sf semiInvariants}, obtain the set $B_6:=\{\left(C_{o,1},1\right)\}$ of coefficients of an invariant of degree $6$.
      \item Let $P_{o,6}(X_1,X_2,X_3)$ be the polynomial with coefficients $C_{o,1}$ in the monomials $\mathbf{X}^{\mathbf{m}_{i}}/\binom{m}{\mathbf{m}_i}$. Let $a$ be a variable and set $P_{a,6}:=aP_{o,6}$, $P_{a,12}:=-H(P_{a,6})/20250$, $P_{a,30}:=HB(P_{a,6},P_{a,12})/23400$ and $P_{a,45}=J(P_{a,6},P_{a,12},P_{a,30})/4860$.
      \item Set $a$ such that
      {\footnotesize
      \begin{align*}
         19683  P_{a,45}^2 = &  4  P_{a,6}^{13}  P_{a,12} + 80  P_{a,6}^{11}  P_{a,12}^2 + 816  P_{a,6}^9  P_{a,12}^3+ 18  P_{a,6}^{10}  P_{a,30} + 4376  P_{a,6}^7  P_{a,12}^4 \\
          & + 198  P_{a,6}^8  P_{a,12}  P_{a,30} + 13084  P_{a,6}^5  P_{a,12}^5 + 954  P_{a,6}^6  P_{a,12}^2  P_{a,30} + 12312  P_{a,6}^3  P_{a,12}^6\\
          & - 198  P_{a,6}^4  P_{a,12}^3  P_{a,30} + 5616  P_{a,6}  P_{a,12}^7 - 162  P_{a,6}^5  P_{a,30}^2 - 5508  P_{a,6}^2  P_{a,12}^4  P_{a,30}\\
          & - 1944  P_{a,6}^3  P_{a,12}  P_{a,30}^2 - 1944  P_{a,12}^5  P_{a,30} - 1458  P_{a,6}  P_{a,12}^2  P_{a,30}^2 + 729  P_{a,30}^3.
      \end{align*}
      }
      \item Let $f$ be a generic function. For $i=6,12,30$, obtain $Q_i(f)$ from \underline{\sf gaugeEquivalentValue} with the invariant $P_{a,i}$, degree $m=i$, order $r=1$, and exponential part $1$.
      \item Let $Q(f)$ be an irreducible factor in $\mathbb{Q}(z)[f]$ of the numerator of $Q_6(f)$, let $s(z):= [3Q_{30}(f)]^2/(8[Q_{12}(f)]^5)$, and $C:=C_f^{-1}C_h$ with $h=\sqrt[12]{Q_{12}(f)}$
      \item \textsf{Return} The ${}_3F_2(-1/60,11/60,7/12;1/2,3/4| z)$ equation, $s(z)$, $Q(f)$ and $C$.
   \end{enumerate}

   \noindent\underline{\sf pullbackSolutionH216}\\
   \textsf{Input:} An equation \eqref{eq0} with differential Galois group isomorphic to $H_{216}^{SL_3}$.
   \begin{enumerate}[leftmargin=*]
      \item Let $y_1,y_2,y_3$ be a basis of solutions of \eqref{eq0} in power series.
      \item Using \underline{\sf semiInvariants}, obtain the set $B_6:=\{\left(C_{o,1},\sqrt[3]{p_1(z)}\right)\}$ of coefficients of a semi-invariant of degree $6$, order $3$, and dimension $d=1$, together with its hyperexponential parts; and $B_9:=\{\left(C_{o,2},1\right)\}$, the invariant of degree $9$.
      \item Let $P_{o,6}(X_1,X_2,X_3)$ and $P_{o,9}(X_1,X_2,X_3)$ be the polynomials with coefficients $C_{o,1}$ and $C_{o,2}$ in the monomials $\mathbf{X}^{\mathbf{m}_{i}}/\binom{m}{\mathbf{m}_i}$. Let $a$ and $b$ be variables and set $P_{a,6}:=aP_{o,6}$, $P_{b,9}:=bP_{o,9}$, $P_{a,12}:=-H(P_{a,6})/108000$, $P_{a,b,12}:=HB(P_{a,6},P_{b,9})/(675P_{a,12})$.
      \item Set $a$ and $b$ such that $864P_{b,9}P_{a,12}=-H(P_{b,9})$ and 
      {\footnotesize
      \begin{align*}
         6912  P_{a,12}^3 = & 1728 P_{a,6}^3 P_{b,9}^2-3 P_{a,6}^2 P_{a,b,12}^2+2592 P_{a,6} P_{a,b,12} P_{b,9}^2+186624 P_{b,9}^4-4P_{a,b,12}.
      \end{align*}
      }
      \item Let $f$ be a generic function. For $i=6,9,12$, obtain $Q_i(f)$ from \underline{\sf gaugeEquivalentValue} with the semi-invariant $P_{a,6}$, $P_{b,9}$, or $P_{a,b,12}$; degree $m=i$; order $r=3$,$1$, or $3$ and exponential part $\sqrt[r_1]{p_1(z)}$, $1$, or $\sqrt[r_1]{p_1(z)^2}$ respectively, depending on whether $i=6$, $9$, or $12$.
      \item Let $Q(f)$ be an irreducible factor in $\mathbb{Q}(z)[f]$ of the numerator of $Q_6(f)$, let $s(z):= [6^6Q_{9}(f)]^4/([Q_{12}(f)]^3)$, and $C:=C_f^{-1}C_h$ with $h=\sqrt[9]{Q_{9}(f)}$.
      \item \textsf{Return} The ${}_3F_2(17/36,2/9,-1/36;1/3,2/3| 1/z)$ equation, $s(z)$, $Q(f)$ and $C$.
   \end{enumerate}

   \noindent\underline{\sf pullbackSolutionH72}\\
   \textsf{Input:} An equation \eqref{eq0} with differential Galois group isomorphic to $H_{72}^{SL_3}$.
   \begin{enumerate}[leftmargin=*]
      \item Let $y_1,y_2,y_3$ be a basis of solutions of \eqref{eq0} in power series.
      \item Using \underline{\sf semiInvariants}, obtain the set $B_6:=\{\left(C_{o,1},\sqrt[2]{p_1(z)}\right)\}$ of coefficients of a semi-invariant of degree $6$, order $2$, and dimension $d=1$, together with its hyperexponential parts.
      \item Let $P_{o,6}(X_1,X_2,X_3)$ be the polynomial with coefficients $C_{o,1}$ in the monomials $\mathbf{X}^{\mathbf{m}_{i}}/\binom{m}{\mathbf{m}_i}$. Let $P_{1,3}(X_1,X_2,X_3)$ and $P_{2,3}(X_1,X_2,X_3)$ be factors of $P_{o,6}(X_1,X_2,X_3)$ in $\mathbb{Q}[\sqrt{3}](X_1,X_2,X_3)$. Let $a$ and $b$ be variables and set $P_{a3}:=aP_{1,3}$, $P_{b,3}:=bP_{2,3}$, $P_{a,b,6}=\left(HB(P_{a,3},P_{b,3})+HB(P_{b,3},P_{a,3})\right)$, $P_{a,b,6,*}:=P_{a,3}P_{b,3}$, $P_{a,b,9}:=\sqrt{3}J(P_{a,3},P_{b,3},P_{a,b,6})/1944$ and $P_{a,b,12}:=-H(P_{a,b,6})/108000$.
      \item Set $a$ and $b$ such that $H(P_{a,3})=-108(2+\sqrt{3})P_{b,3}$ and $H(P_{b,3})=-108(2-\sqrt{3})P_{a,3}$. If $P_{a,b,6}$, $P_{a,b,6,*}$, $P_{a,b,9}$, and $P_{a,b,12}$ doesn't verify
      {\footnotesize
      \begin{align*}
      0 = & (P_{a,b,6}^3 - 3 P_{a,b,6} P_{a,b,6,*}^2 - 2 P_{a,b,6,*}^3 - 432 P_{a,b,9}^2)\\
       & (P_{a,b,6}^3 - 3 P_{a,b,6} P_{a,b,6,*}^2 + 2 P_{a,b,6,*}^3 - 432 P_{a,b,9}^2) + 72 P_{a,b,12} (P_{a,b,6}^4 \\
       &  - 3 P_{a,b,6}^2 P_{a,b,6,*}^2 + 2 P_{a,b,6,*}^4 + 18 P_{a,b,6}^2 P_{a,b,12} - 432 P_{a,b,6} P_{a,b,9}^2 - 24 P_{a,b,6,*}^2 P_{a,b,12})
      \end{align*}
      }
      then set $a$ and $b$ such that $H(P_{a,3})=-108(2-\sqrt{3})P_{b,3}$ and $H(P_{b,3})=-108(2+\sqrt{3})P_{a,3}$
      \item Let $f$ be a generic function. For $i=6,9,12$, obtain $Q_i(f)$ from \underline{\sf gaugeEquivalentValue} with the invariants $P_{a,b,6}$, $P_{b,b,9}$, $P_{a,b,12}$. Obtain $Q_{6,*}$ with the semi-invariants $P_{a,b,6,*}$ of degree $6$, order $2$ and hyperexponential part $\sqrt[r_1]{p_1(z)}$.
      \item Let $Q(f)$ be an irreducible factor in $\mathbb{Q}[\sqrt{3}](z)[f]$ of the numerator of $Q_6(f)$, let $s(z):= [6^6Q_{9}(f)]^4/([Q_{6*}^2-12Q_{12}(f)]^3)$, and $C:=C_f^{-1}C_h$ with $h=\sqrt[9]{Q_{9}(f)}$
      \item \textsf{Return} The ${}_3F_2(17/36,2/9,-1/36;1/3,2/3| 1/z)$ equation, $s(z)$, $Q(f)$ and $C$.
   \end{enumerate}

   \noindent\underline{\sf pullbackSolutionF36}\\
   \textsf{Input:} An equation \eqref{eq0} with differential Galois group isomorphic to $F_{36}^{SL_3}$.
   \begin{enumerate}[leftmargin=*]
      \item Let $y_1,y_2,y_3$ be a basis of solutions of \eqref{eq0} in power series.
      \item Using \underline{\sf semiInvariants}, obtain the set $B_3:=\{\left(C_{o,1},\sqrt[4]{p_1(z)}\right)$, $\left(C_{o,2},\sqrt[4]{p_2(z)}\right)\}$ of coefficients of two semi-invariants of degree $3$, order $4$, and dimension $d=2$, together with their hyperexponential parts.
      \item Let $P_{1,3}$ and $P_{2,3}$ be the polynomials with coefficients $C_{o,1}$ and $C_{o,2}$ in the monomials $\mathbf{X}^{\mathbf{m}_{i}}/\binom{m}{\mathbf{m}_i}$. Let $a$ and $b$ be variables and set $P_{a,3}:=aP_{1,3}$, $P_{b,3}:=bP_{2,3}$, $P_{a,b,6}=\left(HB(P_{a,3},P_{b,3})+HB(P_{b,3},P_{a,3})\right)$ and $P_{a,b,9}:=\sqrt{3}J(P_{a,3},P_{b,3},P_{a,b,6})/1944$.
      \item Set $a$ and $b$ such that $H(P_{a,3})=-108(2+\sqrt{3})P_{b,3}$ and $H(P_{b,3})=-108(2-\sqrt{3})P_{a,3}$. If $P_{a,3}$, $P_{b,3}$, $P_{a,b,6}$, and $P_{a,b,9}$ doesn't verify
      {\footnotesize
      \begin{align*}
         0 = & 2(-7 + 4 \sqrt{3}) P_{a,3} P_{b,3}^5 + \sqrt{3} P_{a,3}^4 P_{a,b,6} - 6 \sqrt{3} P_{a,3}^2 P_{a,b,6} P_{b,3}^2\\
          &  + (12 - 7 \sqrt{3}P_{b,3}^4 P_{a,b,6})- 2 P_{a,3}^5 P_{b,3} + 12 P_{a,3}^2 P_{b,3}^2 P_{a,b,6}\\
          &  + (-8 + 4 \sqrt{3}) P_{a,b,6}^3 + 1728(2 - \sqrt{3}) P_{a,b,9}^2
      \end{align*}
      }
      then set $a$ and $b$ such that $H(P_{a,3})=-108(2-\sqrt{3})P_{b,3}$ and $H(P_{b,3})=-108(2+\sqrt{3})P_{a,3}$.
      \item Let $f$ be a generic function. For $i=3,6,9$, obtain $Q_i(f)$ from \underline{\sf gaugeEquivalentValue} with the semi-invariant $P_{b,3}$, $P_{a,6}$ or $P_{a,b,9}$; degree $m=i$; order $r=4$,$1$, or $1$; and exponential part $\sqrt[r_1]{p_1(z)}$, $1$, or $1$ respectively, depending on whether $i=3$, $6$, or $9$.
      \item Let $Q(f)$ be an irreducible factor in $\mathbb{Q}[\sqrt{3}](z)[f]$ of the numerator of $Q_{3}(f)$, let $s(z):= [Q_{6}(f)]^3/([Q_{6}(f)]^3-432[Q_{9}(f)]^2)$, and $C:=C_f^{-1}C_h$ with $h=\sqrt[6]{Q_{6}(f)}$
      \item \textsf{Return} The ${}_3F_2(-1/12,1/6,5/12;1/4,3/4| 1/z)$ equation, $s(z)$, $Q(f)$ and $C$.
   \end{enumerate}

   \noindent\underline{\sf pullbackSolutionA5}\\
   \textsf{Input:} An equation \eqref{eq0} with differential Galois group isomorphic to $A_5$ or $A_5\times C_3$.\\
   \textsf{Output:} A second order linear ordinary differential equation $E$ with coefficients in a quadratic extension of $\mathbb{Q}(z)$ and differential Galois group $A_5^{SL_2}$, a matrix $C\in M_3(\mathbb{Q}(z,f))$, and the minimal polynomial $Q(f)\in\mathbb{Q}(z)$ of $f$, such that $Y=CY_0$ is a Wronskian of solutions of \eqref{eq0} whenever $Y_0$ is the Wronskian of $y_1^2,y_1y_2,y_2^2$ and $y_1,y_2$ are solutions of $E$.\\
   \begin{enumerate}[leftmargin=*]
      \item Let $y_1,y_2,y_3$ be a basis of solutions of \eqref{eq0} in power series.
      \item Using \underline{\sf semiInvariants}, obtain the set $B_2:=\{\left(C_{o,1},\sqrt[r_1]{p_1(z)}\right)\}$ of coefficients of a semi-invariant of degree $2$, order $r_1|3$, and dimension $d=1$, together with its hyperexponential parts.
      \item Let $P_{o,2}$ be the polynomial with coefficients $C_{o,1}$ in the monomials $\mathbf{X}^{\mathbf{m}_{i}}/\binom{m}{\mathbf{m}_i}$.
      \item Let $f$ be a generic function. Obtain $Q_2(f)$ from \underline{\sf gaugeEquivalentValue} with the semi-invariant $P_{0,2}$, degree $m=2$, order $r=r_1$, and exponential part $\sqrt[r_1]{p_1(z)}$.
      \item Let $Q(f)$ be an irreducible factor in $\mathbb{Q}(z)[f]$ of the numerator of $Q_2(f)$.
      \item Let $L_f=\partial z^3-b_2\partial z^2-b_1\partial z-b_0$ where $[b_0,b_1,b_2]$ are the entries of the third row of $A_{C_f}=C_f'C_f^{-1}+C_fAC_f^{-1}$. Let $L_0$ be such that $L_f=Sym^2(L_0),$
      $$L_0=\partial z^2-b_2/3\partial z-b_2^2/18+(1/12)b_2'-b_1/4.$$
      \item \textsf{Return} $L_0y=0$, $C=C_f^{-1}$, $Q(f)$.
   \end{enumerate}

Note that in \underline{\sf pullbackSolutionA5}, the solutions of $L_f$ satisfy a homogeneous quadratic equation with constant coefficients and therefore the operator is a symmetric square \cite{SINGER1988}.

   \subsection{The algorithm}

   \noindent\underline{\sf kovacicOrder3}\\
   \textsf{Input:} A third order differential irreducible operator $L=Dz^3+a_2Dz^2+a_1Dz+a_0$ with coefficients in $\mathbb{Q}(z)$.\\
   \textsf{Output:} A basis of Liouvillian solutions of $L$.\\
   \begin{enumerate}[leftmargin=*]
      \item Factorize $L$. If $L=L_1L_2$ then let $S$ be basis of Liouvillian solutions of $L$, obtained using reduction of order and Kovacic algorithm for second order differential operators.
      \item Else $L$ is irreducible. 
         \begin{enumerate}[leftmargin=5pt] 
            \item Let $L_0 := \hbox{Sym}(L,Dz+a_{2}/3)$. If $L_0$ has an irregular singularity, set\\ $S_0:=\emptyset$.
            \item Else $L$ is Fuchsian. Denote $G=\hbox{Gal}(L_0)$ 
               \begin{enumerate}[leftmargin=5pt] 
                  \item Let $B_3:=\text{\underline{\sf semiInvariants}}(L_0,m=3,r=2,d=2)$. If $B_3$ is non-empty then $G$ is imprimitive, set\\ $S_0:=\text{\underline{\sf riccatiSolutionImprimitive}}(L_0)$.
                  \item Else $B_3$ is empty and $G$ is primitive.
                     \begin{itemize}[leftmargin=5pt]
                        \item Let $B_2:=\text{\underline{\sf semiInvariants}}(L_0,m=2,r=3,d=1)$. If $B_2$ is non-empty, then:
                        \begin{itemize}[leftmargin=5pt]
                           \item Let $B_6:=\text{\underline{\sf semiInvariants}}(L_0,m=6,r=1,d=1)$. If $B_6$ has only one element then $G\simeq SO_3(\mathbb{C})$, set\\ $S_0:=\emptyset$.
                           \item Else $B_6$ has two elements, $G$ is isomorphic to $A_5$ or $A_5\times C_3$, set \\ $S_0:=\text{\underline{\sf pullbackSolutionA5}}(L_0)$. 
                        \end{itemize}
                        \item Else $B_2$ is empty. 
                        \begin{itemize}[leftmargin=5pt]
                           \item Let $B_3:=\text{\underline{\sf semiInvariants}}(L_0,m=3,r=4,d=2)$. If $B_3$ is not empty then $G\simeq F_{36}^{SL_3}$, set\\ $S_0:=\text{\underline{\sf pullbackSolutionF36}}(L_0)$.
                           \item Else $B_3$ is empty.
                              \begin{itemize}[leftmargin=5pt]
                                 \item Let $B_4:=\text{\underline{\sf semiInvariants}}(L_0,m=4,r=3,d=1)$. If $B_4$ is not empty then $G$ is isomorphic to $G_{168}$ or $G_{168}\times C_3$, set\\ $S_0:=\text{\underline{\sf pullbackSolutionG168}}(L_0)$.
                                 \item Else $B_4$ is empty.
                                    \begin{itemize}[leftmargin=5pt]
                                       \item Let $B_6:=\text{\underline{\sf semiInvariants}}(L_0,m=6,r=2,d=1)$. If $B_6$ has more than 1 element then $G\simeq H_{72}^{SL_3}$, set\\ $S_0:=\text{\underline{\sf pullback\_solution\_H72}}(L_0)$.
                                       \item Else, if $B_6$ has exactly one element then $G\simeq A_{6}^{SL_3}$, let $S_0:=$\underline{\sf pullbackSolutionA6}$(L_0)$.
                                       \item Else $B_6$ is empty.
                                          \begin{itemize}[leftmargin=5pt]
                                             \item Let $B_6:=\text{\underline{\sf semiInvariants}}(L_0,m=6,r=3,d=1)$. If $B_6$ is not empty then $G\simeq H_{216}^{SL_3}$, set\\ $S_0:=\text{\underline{\sf pullbackSolutionH216}}(L_0)$.
                                             \item Else $G\simeq SL_3$, let $S_0=\emptyset$.
                                          \end{itemize}  
                                    \end{itemize}
                              \end{itemize}
                        \end{itemize}
                     \end{itemize}  
               \end{enumerate}
            \item From $S_0$ construct a basis of solutions $y_1,y_2,y_3$ and set $S:=\{e^{\frac{1}{3}\int a_2dz}y_i|\ i=1,2,3\}$.
         \end{enumerate}
      \item \textsf{Return} $S$.
   \end{enumerate}

\section{Examples}

The complexity of computing rational and hyperexponential solutions, or invariants and semi-invariants, of a linear ordinary differential equations depends on the degree, the order and on the local exponents at the singularities \cite{BARKATOU1999,VANHOEIJ1997}, and therefore the same applies to any algorithm that relies on their computation, like Kovacic's algorithm or ours. Because of this, we do not aim at obtaining low complexity in terms of the degree, order and height of the coefficients. The objective is rather to obtain a workable algorithm on reasonable examples rather than on cases of worst possible complexity. The most expensive part is obtaining all the entries of the semi-invariants and their hyperexponential values and so our algorithm was designed to optimize this part.

We present four examples illustrating the reach of our algorithm. The computations were done on a Macbook Pro 2018 2.6 Ghz with 16GB of RAM in Maple 2016.

\textbf{An imprimitive case.} Consider the operator
{\small
$$
\partial_z^3+\dfrac{1}{2}\dfrac{7z-4}{z(z-1)}\partial_z^2+\dfrac{1}{27}\dfrac{41z-6}{z^2(z-1)}\partial_z+\dfrac{2}{729}\dfrac{1}{z^2(z-1).}
$$
}
In 0.430s, the algorithm produces the solutions $e^{\int \omega(z)dz}/(z^{2/3}(z-1)^{1/2})$ where
{\small
\begin{align*}
   0 & =\omega(x)^3-\frac{1}{2}\frac{7z-4}{z(z-1)}\omega(x)^2+\frac{1}{108}\frac{440z^2-503z+144}{z^2(z-1)^2}\omega(x)\\
   & -\frac{1}{5832}\frac{9196z^3-15773z^2+9034z-1728}{z^3(z-1)^3}
\end{align*}
}

\textbf{A $G^{SL_3}_{168}$ test case.} In \cite{SANABRIA2022}, it is reported that to obtain a solution to the operator
{\small
$$
\partial_z^3+\frac{1}{2}\frac{7z-4}{z(z-1)}\partial_z^2+\frac{1}{252}\frac{387z-56}{z^2(z-1)}\partial_z-\frac{1}{74088}\frac{85}{z^2(z-1)}
$$}
in terms of ${}_3F_2(-1/42,5/42,17/42;1/3,2/3| z)$, gauge transformations and pullbacks, it required more than 145000s of computations in a similar computer to ours. Our algorithm produces the solutions in 59.611s.

\textbf{A $G^{SL_3}_{168}$ case.} The operator
{\small
\begin{align*}
   \partial_z^3 & + \frac{31z^7-17}{z(z^7-1)}\partial_z^2+\frac{1}{4}\frac{1152z^{14}-1513z^7+312}{z^2(z^7-1)^2}\partial_z\\
    & +\frac{1}{8}\frac{6336z^{21}-13735z^{14}+8805z^7-720}{z^3(z^7-1)^3}   
\end{align*}
}
has $9$ non-apparent singularities. Without Proposition \ref{prop_d}, it would be necessary to consider the splitting field of $a_3$ for each possible Galois group. This is significantly longer as there are $128$ candidates for the hyperexponential part, instead of $4$ when we can use the case $d=1$, which allows to avoid these field extensions. Our algorithm yields a solution in 37.517s.

\textbf{An $A_6$ case.} The operator
{\small
\begin{align*}
   \partial_z^3 & +\frac{1}{1200}\frac{1088z^2-1313z+1125}{z^2(z-1)^2}\partial_z\\
    & -\frac{1}{21600}\frac{19456z^3-35195z^2+52189z-20250}{z^3(z-1)^3}   
\end{align*}
}
has Galois group $A_6$. Computing its solutions with hypergeometric functions, gauge equivalences and pullbacks requires computing an invariant of degree $30$. Computing its solution with a Riccati polynomial requires obtaining the invariant of degree $45$. Our algorithm computes the solution in 1230.451s.

\section{Conclusion}
 We produced an algorithm for computing Liouvillian solutions of linear ordinary differential equations with rational coefficients whose execution time is manageable on reasonable examples. There are three key elements in the development of the algorithm. The first one is that the semi-invariants of all the different possible solvable groups in $SL_3(\mathbb{C})$ that arise as Galois groups of third order differential operators are related via the hessian, the bordered hessian and the jacobian. From this, it's possible to obtain all the semi-invariants computing only semi-invariants of low order. This represent an important improvement since the computation of the semi-invariants of higher order (for example, the invariant of order $45$ for $A_6$ or the one of order $21$ for $G_{168}$) are out of reach in a reasonable time with previous methods. The second key element is that we simplify the computation of the semi-invariant by reducing the number of candidates for the exponential part. In all generality, the number of candidates for the exponential part is big. The size of this set is a well know drawback in Beke's algorithm \cite{VANHOEIJ1997}. We managed to overcome this difficulty as in most cases, no extension is necessary, and that at worst, a quadratic extension will be required. This happens in semi-invariants of degree $3$ and order $2$ or $4$; in the rest of the cases, we can obtain the semi-invariants by staying in the field of rational functions. The third key element is computing the values of the semi-invariants in gauge equivalent equations and the Riccati polynomials using symmetric powers, together with the test in Propostion \ref{prop_linear} to identify semi-invariants of order $3$ that factors into linear terms. These three improvements renders the approach suggested in \cite{VANHOEIJ1999} into an efficient algorithm.

%
\bibliographystyle{ACM-Reference-Format}
\bibliography{K3}

%

\end{document}